\documentclass[10pt,reqno]{amsart}

\usepackage{amssymb}
\usepackage{xypic}
\xyoption{all}
\def\umono{\ar@{_{(}->}[u]}
\def\uumono{\ar@{_{(}->}[uu]}

\def\lmono{\ar@{_{(}->}[l]}
\def\llmono{\ar@{_{(}->}[ll]}

\newcommand{\Z}{{\mathbb Z}}

\newcommand{\Hom}[2]{\operatorname{Hom}(#1,#2)}

\newcommand{\A}{\ifmmode{\mathcal{A}}\else${\mathcal{A}}$\fi}
\newcommand{\K}{\ifmmode{\mathcal{K}}\else${\mathcal{K}}$\fi}
\newcommand{\U}{\ifmmode{\mathcal{U}}\else${\mathcal{U}}$\fi}
\newcommand{\T}{\ifmmode{\mathcal{T}}\else${\mathcal{T}}$\fi}
\newcommand{\FF}{\ifmmode{\mathcal{F}}\else${\mathcal{F}}$\fi}
\newcommand{\LL}{\ifmmode{\mathcal{L}}\else${\mathcal{L}}$\fi}

\newtheorem{theorem}{Theorem}[section]
\newtheorem{proposition}[theorem]{Proposition}
\newtheorem{corollary}[theorem]{Corollary}
\newtheorem{lemma}[theorem]{Lemma}
\newtheorem{definition}[theorem]{Definition}
\newtheorem{remark}[theorem]{Remark}
\newtheorem{example}[theorem]{Example}

\title[Conditionally flat functors]{Conditionally flat functors on spaces and groups}

\author{Emmanuel Dror Farjoun}
\address{Einstein Institute of Mathematics\\
Hebrew University of Jerusalem (Gig'at Ram)\\
Jerusalem 91904, Israel}
\email{farjoun@math.huji.ac.il}

\author{J\'er\^{o}me Scherer}
\address{EPFL \\ SB MATHGEOM\\ Station 8, MA B3 455
CH-1015 Lausanne, Switzerland}
\email{jerome.scherer@epfl.ch}

\thanks{The second author was supported by FEDER/MEC grant MTM2007-61545}

\date{\today}
\begin{document}


\begin{abstract}
Consider an extension of groups $1 \rightarrow K \rightarrow G \rightarrow Q \rightarrow 1$ which enjoys the
property that the quotient by the lower central series $\Gamma_{c+1}$ produces another extension
$1 \rightarrow K/ \Gamma_{c+1} K \rightarrow G /\Gamma_{c+1} G \rightarrow Q /\Gamma_{c+1} Q \rightarrow 1$,
of nilpotent groups of class~$c$. We say that the extension is $\Gamma_{c+1}$-flat.
Let us pull back the original extension along any homomorphism $X \rightarrow G$.
Does the pullback extension enjoy the same $\Gamma_{c+1}$-flatness property?

To answer this question we consider not only quotients by the lower central series, but any localization
functor in the category of groups. In fact we start by studying the analogous question for spaces, where we
replace extensions by fibration sequences. We prove that the only homotopical localization functors which
behave well under pull-backs are nullifications. In the category of groups, nullifications also enjoy this
property, and so do all epireflections arising from a variety of groups. In particular the answer to the 
question about the nilpotent quotients is positive.
\end{abstract}


\maketitle


\section*{Introduction}
\label{sec intro}
This work originates in the following question. Consider a fibration
$F \to E \to B$ and pull it back along a map $X \to B$. Can this new fibration $F \to P \to X$  be
``more complicated" than the original one? Often the answer is negative, and
indeed, pulling back fibrations (or extensions of groups) can only simplify them from several  points of view.
For example,  if the first fibration
has  a section, or induces a fibration on the $n$-th Postnikov stage, then so would one obtained by a pull-back along any map.

The  properties of fibrations and short exact sequences we consider in this article are related to (homotopical) localization.
Given a specific localization functor $L$, we are interested in the  ``flatness" property of a fibration  $F\to E\to B$ namely, 
that of being  preserved as such by this
localization functor $L$. This was  considered  to some extend in \cite{MR1997044}. So we say that a fibration sequence $F \to E \to B$ over
a connected base space is \emph{$L$-flat} if the sequence $LF \to LE \to LB$ is also a homotopy fibration sequence.
A classical example is the fibre lemma of Bousfield and Kan in \cite{MR51:1825}.
This lemma asserts  the preservation of principal fibrations with a connected fibre by the homological completion functors $R_\infty$,
for any commutative ring $R$. To which extent localization functors preserve principal fibrations is also the subject of \cite{MR2581908}.
More generally, Bousfield, \cite{MR1257059}, the first author and Smith, \cite{MR1318881}, analyzed the ``error term" calculating the
failure of flatness.


\medskip

In general however nice properties are not preserved under pullback.  We exhibit elementary counter-examples
in homotopy theory and group theory, see Example~\ref{ex:homology} and Theorem~\ref{thm:epireflection}.
For homotopical  localization functors   $L=L_f$  one can understand the situation
as follows. A functor  $L$ is said to be \emph{conditionally flat} if any
pull back of an $L$-flat fibration is again $L$-flat. One direction of the following result has been shown in~\cite{MR1997044},
and many ideas used here are explicitly or implicitly present in that article.

\medskip

\noindent{\bf Theorem~\ref{thm:nullification}.}
{\it 
A homotopy localization functor $L$ is conditionally flat if and only if $L$ is a nullification functor $P_A$ for some space $A$.
}

\medskip

In other words,  if we can conclude that, given an $L$-flat fibration sequence, so is any pullback
of this fibration, then the localization functor $L$ must be a nullification functor and, in that case, any pullback of an
$L$-flat fibration sequence is $L$-flat. The nullification functor $P_A$ kills $A$, and all spaces constructed from
$A$ by push-outs, wedges, telescopes, and extensions by fibrations, \cite{MR97i:55023}. Typical examples are Postnikov sections and
Quillen's plus-construction, \cite{Dror}.

\medskip

We then  turn to group theory where one replaces fibration sequences by group extensions, namely
by short exact sequences  $1\to K \to E \to G\to 1$ and consider again  flatness of localization functors in relation to  pullbacks 
along  group homomorphisms $H\to G$. It turns out that the situation is more involved here  and 
in particular the answer is more interesting since there are localization functors which are not nullification functors  
for which $L$-flatness is preserved by pullbacks of short exact sequences.

This is  easily  seen to be the case for any right exact functor such as  the abelianization  functor $G\to G_{ab}$.  In fact, consider
the quotient $L_cG = G/\Gamma_c(G)$ by the $c$-th term in the lower central series, turning a group $G$ into a nilpotent one
of some fixed class $c$. Flatness with respect to this functor  is a property which behaves well with taking pull-backs. 
We prove in fact that any localization defined by a variety of groups, \cite{MR0215899}, shares this feature.

\medskip

\noindent{\bf Theorem~\ref{thm:variety}.}
{\it 
Let $\mathcal W$ be any variety of groups. The asssociated localization functor $L$ in the category of groups
is then right exact and thus conditionally flat.
}

\medskip

The  classifying space construction yields fibrations of spaces and translates the question into the homotopy category.
There are thus homotopical localization functors which are not nullification functors, but nevertheless preserve \emph{certain}
$L$-flat fibration sequences under pull-backs. This is shortly discussed in Remark~\ref{rem:classify}.

\medskip

\noindent {\bf Organization and content}
The rest of the paper is organized as follows:  The first section  gives  basic definitions and notations. 
The second contains  the main result about fibration sequences and their conditional preservation by functors. 
In the third section flatness of functors on groups is considered, here right exact functors are shown to be conditionally flat. 
The last section deals with  (counter-)examples and possible further developments.

\medskip

\noindent {\bf Acknowledgements.} This work started when the first
author visited the EPFL in Lausanne and the facilitation of this working visit was greatly appreciated.

\section{Notation and terminology}
\label{sec:notation}
We are interested in properties of fibration sequences of pointed 
spaces (or  
simplicial sets) and extensions of groups. As they share
many common features we will introduce some terminology which applies to both settings.

We will work with homotopy localization functors $L$ in the sense of Bousfield, see \cite{MR57:17648} and \cite{Dror}, 
in the category of  pointed spaces or groups. In practice
we fix a map $f$ of spaces or groups and consider the localization functor $L_f$ which inverts $f$, \cite{Dror}.
Instead of defining $L$-flatness only for fibrations as we did in the introduction for simplicity, we do it for any map.

\begin{definition}
\label{def:Lflat}
{\rm Let $L$ be a homotopy  functor on spaces. A map $E \to B$ is $L$-\emph{flat} if the canonical comparison map
$L Fib(E \to B) \to Fib(LE \to LB)$ is a weak equivalence.
}
\end{definition}

Equivalently, a fibration sequence $F \rightarrow E \rightarrow B$ is $L$-flat if and only if
$LF \rightarrow LE \rightarrow LB$ is again a fibration sequence. It will be convenient sometimes
to work with fibrations rather than with maps of which we have to take the homotopy fiber.
The same terminology applies to groups, but only for extensions. 
Hence, a group extension $1 \rightarrow N \rightarrow E \rightarrow G \rightarrow 1$ 
is $L$-flat if $1 \rightarrow LN \rightarrow LE \rightarrow LG \rightarrow 1$ is again an extension of groups.

\begin{example}
\label{ex:Pflat}
{\rm Let $P$ be a nullification functor, i.e. $P = L_f$ for $f: A \rightarrow *$. Then any map $E \rightarrow B$
over a $P$-local base space $B$ is $P$-flat, \cite[Corollary~D.3]{Dror}.
}
\end{example}

Most maps are not $L$-flat for a given localization functor $L$. The question we ask is about the preservation
of flatness under base change, that is, if we happen to work
with an $L$-flat map, we ask whether the map obtained by pulling back along an arbitrary map to the base
is $L$-flat again.

\begin{definition}
\label{def:condiflat}
{\rm The map $E \rightarrow B$ is \emph{fully  $L$-flat} if
all its  pullbacks are $L$-flat. A functor  $L$ is \emph{conditionally flat} if any
pull back of an $L$-flat map is again $L$-flat, i.e. if any $L$-flat map is fully
$L$-flat.
}
\end{definition}

Thus full $L$-flatness  refers always to  a map  and  means both  its $L$-flatness (because one can choose to pull-back along
the identity map) and the  $L$-flatness of any  of its pullbacks. Namely, for any map $B'\to B$ the map 
$E \times_B B' \rightarrow B'$ is $L$-flat. As mentioned in the introduction, the fibre lemma of Bousfield and Kan
states that all principal fibrations with  a connected fibre-group are $R_\infty$-fully flat.

\medskip

The main players here are  not the various maps we consider but rather
the  functors $L$. Conditional flatness refers to functors:  In general, localization functors  are not flat, i.e. they
do not preserve all fibration sequences but often, if a map is $L$-flat, then so is any pull-back.
This is a property of the functor $L$ which we call here ``conditionally flat."

\medskip

We will use the same terminology for group extensions and group theoretic localization functors. In fact the above definitions
make sense not only for localization functors, but arbitrary endofunctors (cellularization functors for example, but possibly also
non-idempotent ones like the James construction $JX \simeq \Omega \Sigma X$, or the infinite symmetric product $SP^\infty X$).

\section{ Localization of  fibration sequences}
\label{sec:fibrations}
We extend in this section the results of Berrick and the first author in~\cite{MR1997044}.
It was shown there that nullification functors are always conditionally flat and we prove now
that, in fact, a homotopy localization functor $L$ is conditionally flat if and only if it is a nullification
functor. Parts of our arguments resemble those in~\cite{MR1997044}, but we prefer to include a
complete proof because we will follow the precise same steps in the next section for group theoretic localizations.

\begin{example}
\label{ex:Postnikov}
{\rm We choose for $f$ the map collapsing a sphere $S^{n+1}$ to a point, so that the localization $X \to L_f X = P_{S^{n+1}} X$ is homotopy equivalent
to taking the $n$-th Postnikov stage $X \to X[n]$. Saying that a fibration sequence $F \to E \to B$ is ``$n$-Postnikov-flat" amounts to
saying that the connecting homomorphism $\pi_{n+1} B \to \pi_n F$ is trivial. Therefore, if $F \to E \to B$ is $n$-Postnikov-flat,
so is any pullback fibration sequence $F \to E \times_B X \to X$, for any map $X \to B$.

A typical example of a fibration sequence which is not $n$-Postnikov flat is the path-loop fibration on a non $(n+1)$-connected
space such as $S^{n+1}$. The fibration $\Omega S^{n+1} \rightarrow P S^{n+1} \rightarrow S^{n+1}$ is not
$n$-Postnikov flat,   since applying Postnikov
sections destroys the exactness of the sequence of homotopy groups of the spaces involved.}
\end{example}

However, the Postnikov section functors, as well as all nullification functors (and only these!), are  conditionally flat.

\begin{theorem}
\label{thm:nullification}
A homotopy localization functor $L$ is conditionally flat if and only if $L$ is a nullification functor $P_A$ for some space $A$.
\end{theorem}

The proof will be given at the end of the section.
We start with a few reduction steps. The first one allows us to work with maps having local homotopy fibers. Recall
that $L$ is conditionally flat if pulling back any $L$-flat map produces another $L$-flat map.

\begin{lemma}
\label{lem:fiberwise}
Let $L$ be a homotopy localization functor and assume that  any $L$-flat map with
$L$-local homotopy fiber is fully $L$-flat. Then $L$ is conditionally flat.
\end{lemma}

\begin{proof}
Let us consider a pull-back diagram of fibration sequences
\[
\xymatrix{
F \ar[r] \ar@{=}[d] & P \ar[r] \ar[d] & X \ar[d] \\
F \ar[r] & E \ar[r] & B
}
\]
where the bottom map $E \to B$ is $L$-flat. We have to show that so is the top map $P \to X$.
Applying fiberwise localization, \cite[Section 1.F]{Dror} to both fibrations yields a new diagram
\[
\xymatrix{
LF \ar[r] \ar@{=}[d] & \overline P \ar[r] \ar[d] & X \ar[d] \\
LF \ar[r] & \overline E \ar[r] & B
}
\]
together with maps $E \to \overline E$ and $P \to \overline P$ which are $L$-local equivalences. Notice that $\overline P$ is
obtained as the homotopy pull-back of $\overline E \to B \leftarrow X$. By assumption the map $E \to B$ is $L$-flat,
thus so is the fiberwise localization $\overline E \to B$, since applying $L$ to it yields
the map $LE \to LB$, whose homotopy fiber is $LF$. We suppose that this property is preserved by taking pull-backs of fibrations
with local fiber. Therefore we conclude that the map $\overline P \to X$ is $L$-flat, which implies in turn that
$P \to X$ is so.
\end{proof}

The second step reduces the problem to studying fibration sequences of local spaces.

\begin{lemma}
\label{lem:localbase}
Let $L$ be a localization functor and assume that all fibration sequences of $L$-local spaces are fully $L$-flat. 
Then $L$ is conditionally flat.
\end{lemma}

\begin{proof}
We know from the previous lemma that we can assume the fiber to be $L$-local. We consider thus an $L$-flat fibration sequence
$F \rightarrow E \rightarrow B$ where $F$ is $L$-local and a map $g: X \rightarrow B$. We can also assume that $B$ is connected.
We complete it to the following diagram by constructing first
the pullback along $g$ and second, by localizing the bottom row:
\[
\xymatrix{
F \ar[r] \ar@{=}[d] & P \ar[r] \ar[d] & X \ar[d] \\
F \ar[r] \ar@{=}[d] & E \ar[r] \ar[d] & B \ar[d] \\
F \ar[r] & LE \ar[r] & LB
}
\]
Since $B$, and hence $LB$ are connected spaces, we see that $E$ is the homotopy pull-back
of the diagram $LE \to LB \leftarrow B$, and therefore $P$ is the homotopy pull-back of $LE \to LB \leftarrow X$.
We conclude that the top fibration sequence is $L$-flat.
\end{proof}

Our third and last step allows us to perform the pullback construction along a very specific map, namely the localization
map $\eta_X: X \rightarrow LX$.

\begin{lemma}
\label{lem:localizationmap}
Let $L$ be a localization functor and assume that, for any connected space $X$ and any fibration sequence
$F \rightarrow E \rightarrow LX$ of $L$-local spaces, the pullback fibration sequence along $\eta_X: X \to LX$ 
is $L$-flat. Then $L$ is conditionally flat.
\end{lemma}

\begin{proof}
We need only prove by Lemma~\ref{lem:localbase} that a fibration sequence $F \rightarrow E \rightarrow B$ of $L$-local spaces 
is fully $L$-flat. Consider thus any map $\alpha: X \to B$.
We must show that the pull-back fibration sequence $F \rightarrow P \rightarrow X$ is $L$-flat.
Since $\alpha$ factors through the localization map $X \to LX$ we construct a diagram of fibration sequences involving
the $L$-local homotopy  pull-back $Q$ of $E \to B \leftarrow LX$:
\[
\xymatrix{
F \ar[r] \ar@{=}[d] & P \ar[r] \ar[d] & X \ar[d]^{\eta_X} \\
F \ar[r] \ar@{=}[d] & Q \ar[r] \ar[d] & LX \ar[d] \\
F \ar[r] & E \ar[r] & B
}
\]
Since by our assumptions the space $Q$ is a homotopy pull back of  local spaces, it follows
that the top right square is also a homotopy pull-back square and the middle row is a fibration sequence of
$L$-local spaces, \cite[A.8 (e3)]{Dror}. By assumption the top fibration is preserved by $L$.
\end{proof}

We are now ready to prove the main theorem of this section.

\begin{proof}[Proof of Theorem~\ref{thm:nullification}]
We consider a fibration sequence $F \to E \to LB$ of $L$-local spaces and will show that
the pullback along the localization map $\eta_B: B \to LB$ is $L$-flat. We will deduce from Lemma~\ref{lem:localizationmap}
that $L$ is conditionally flat. Let us observe the following diagram:
\[
\xymatrix{
F \ar[r] \ar@{=}[d] & Q \ar[r] \ar[d] & B \ar[d] \\
F \ar[r] & E \ar[r] & LB
}
\]
If $L$ is not a nullification functor, there exists a space $B$ such that the homotopy
fiber $\overline L B$ of the localization map $B \to LB$ is not $L$-acyclic. An explicit example is constructed in \cite[Theorem~2.1, (iv) $\Rightarrow$ (i)]{MR1997044}.
However the fibration sequence $\Omega (LB) \to P(LB) \to LB$
is one of $L$-local spaces. The pull-back fibration sequence $\Omega(LB) \to \overline L B \to B$ is not $L$-flat because the localization
of the total space $L \overline L B$ is not contractible.

\medskip

When $L$ is of the form $P_A$, the homotopy fiber of the localization map $B \to P_A B$ is $P_A$-acyclic. Therefore the fibration sequence
$\overline P_A B \to Q \to E$ is preserved by $P_A$, \cite[Theorem~1.H.1]{Dror}, i.e. $P_A Q \simeq E$ which means that the
fibration sequence $F \rightarrow Q \rightarrow B$ is $P_A$-flat.
\end{proof}

Hence, nullification functors such
as plus-constructions, Postnikov sections, $B\Z/p$-nullification appearing in Miller's work on the Sullivan conjecture, \cite{Miller}, are
all conditionally flat. Counter-examples can now also be easily constructed.

\begin{example}
\label{ex:homology}
{\rm Consider localization $L_{H\mathbf Z}$ with respect to ordinary homology $H^*(-; \mathbb Z)$. There are many spaces for which
the homotopy fiber of the localization are not acyclic, often not even connected. One of the ``smallest" examples  is Whitehead's example, 
\cite[IV.7 Example 3]{MR516508}, of a three cell complex
$X= (S^1 \vee S^2) \cup e^3$ having the homology of a circle. The homological localization map $X \to S^1$ coincides with the first Postnikov
section, so that the homotopy fiber is the universal cover $\tilde X$, a simply connected but non-trivial $H\mathbb Z$-local space. The pull-back
of the path-loop fibration $\mathbb Z \to PS^1 \to S^1$ along the map $X \to S^1$ yields a fibration $\mathbb Z \to \tilde X \to X$ which is not
$L_{H \mathbb Z}$-flat.}
\end{example}

\section{Conditionally flat group-functors and varieties}
\label{sec:group}

We move now to the category of groups, replacing the notion of fibration sequence by short exact sequences. 
Our aim is to look at the notions of flatness and conditional flatness
for functors and extensions of groups. The result we just proved for homotopical localization does not translate directly for groups.
Indeed, we will see in Example~\ref{ex:abelian} below that abelianization is a conditionally flat localization functor
(but not a nullification).
The point of course is that the corresponding homotopical localization of spaces is  not conditionally flat, but it is so on
fibrations which are constructed by applying the classifying space to an extension of groups. We start this section by
proving that group theoretical nullification functors are conditionally flat, even though they are not the only ones.
We notice that the same reduction steps we went through for spaces in Section~\ref{sec:fibrations} do work for groups.

\begin{proposition}
\label{prop:localgroup}
Let $L$ be a localization functor in the category of groups. Assume that, for any group $G$ and any extension of $L$-local
groups $K \to E \to LG$, the pull-back along the localization morphism $\eta_G: G\to LG$ is $L$-flat. Then $L$ is conditionally flat.
\end{proposition}

\begin{proof}
We must show that the pull-back of an $L$-flat extension is in turn $L$-flat.
The first reduction step allowing us to consider only extensions with local kernel is obtained by applying
fiberwise localization to our group extensions. Such a construction
is available for groups thank to work of Casacuberta and Descheemaeker, \cite{ MR2125447}.
The second step reduces to the study of extensions of local groups and this works simply because one recognizes
a pull-back square by comparing the kernels. The third and last step is exactly as in Lemma~\ref{lem:localizationmap}
and permits us to pull-back along a localization map $G \rightarrow LG$.
\end{proof}

To any group homomorphism $\varphi$ one associates an (idempotent, augmented) localization functor $L_\varphi$
in the category of groups, which inverts $\varphi$ in  a universal way. When $\varphi$ is of the form $H \rightarrow \{ e \}$, the
localization is called \emph{nullification} and usually written $P_H$, just like in the homotopical setting. 

\begin{theorem}
\label{thm:nullificationgroup}
Any nullification functor in the category of groups is conditionally flat.
\end{theorem}

\begin{proof}
The key point is that the kernel of the localization morphism $G \to LG$ is $L$-acyclic when (in fact if and only if) $L$ is
a nullification functor, \cite[Proposition~3]{MR1758736}.
\end{proof}

We move now as promised to more ``exotic"  conditionally flat localization functors, that is some which are not nullifications.
Our motivation was to study the interplay of pulling back an extension and taking the quotient by the lower central series.
We are now ready to come back to this question.

\begin{example}
\label{ex:abelian}
{\rm Assume that the group extension $1 \rightarrow K \to E \to G \rightarrow 1$ \emph{abelianizes well}, that is,
the abelianization $0 \rightarrow K_{ab} \to E_{ab} \to G_{ab} \rightarrow 0$
forms an extension (of abelian groups). Then for any morphism $H \to G$, the pull-back extension $K \to P \to H$ also abelianizes
well. In our general terminology, abelianization is conditionally flat.

The argument is simple. Abelianization is right exact, a fact that can be proved either directly, or by using the group
homology five term exact sequence (which can be deduced from Hopf's formula, \cite[Exercise~II.5.6]{MR672956}).
Hence we only need to show that $K_{ab}\to P_{ab}$ is injective. By assumption the extension
$1 \rightarrow K \to E \to G \rightarrow 1$ is ab-flat, hence $K_{ab}\to E_{ab}$
is injective. As it factors through $P_{ab}$ the conclusion follows.
}
\end{example}

The same proof actually applies to \emph{any} right exact functor.

\begin{proposition}
\label{prop:halfexact}
Let $F$ be a right exact functor in the category of groups. Then $F$ is conditionally flat. \hfill{\qed}
\end{proposition}

A \emph{variety of groups} $\mathcal W$ is
defined by a set of words $W$ in a free group $F$ on a countable, infinite set of generators $\{ x_1, x_2, x_3, \dots \}$.
Following \cite{MR0215899}, $\mathcal W$ is the family of all groups $G$ having the property that every homomorphism
from $F$ to $G$ sends the words in $W$ to~$1$. Take $\phi: F \to F/WF$, where $WF$ is the normal subgroup
generated by images of words in $W$ under  all homomorphisms $F \to F$. The localization functor $L_\phi$ sends then a group
$G$ to the largest quotient which belongs to the variety $\mathcal W$, \cite[Proposition~3.1]{MR1676617}. The kernel can
be described as the subgroup $WG$ of $G$ generated by all images of words in $W$ under morphisms from $F$.

\begin{example}
\label{ex:lowercentral}
{\rm For any given integer $c\geq 1$, we take $W$ to be generated by the single word $[ \dots [x_1, x_2 ], \dots x_{c}], x_{c+1}]$,
a $c$-fold commutator.
For any group $G$ the subgroup $WG$ is nothing but $\Gamma_c (G)$ the $c$-th term in the lower central series. Thus, the
localization $L_\phi$ sends $G$ to $G/\Gamma_c(G)$. When $c=1$ for example, $W$ is generated by a single
commutator $[x_1, x_2]$. A group belongs to $\mathcal W$ if and only if it is abelian, the group homomorphism $\phi$ is
$F \rightarrow F/[F, F] = F_{ab}$ and $L_\phi$ is abelianization.
}
\end{example}

In general $W(WG) \neq WG$, as is shown by abelianization (of the dihedral group of order $8$ say). In fact Casacuberta, Rodr\'\i guez and
Scevenels show that $W(-)$ is idempotent if and only if the corresponding localization is a nullification, \cite[Theorem~2.3]{MR1676617}. 
We prove now that any variety
of groups determines a right exact localization functor, hence a conditionally flat functor. We could also have applied our general
principle Proposition~\ref{prop:localgroup} and proven ``by hand" that the pull-back of an extension of local groups
$K \to E \to G/WG$ along the localization map $G \to G/WG$ is flat.


\begin{theorem}
\label{thm:variety}
Let $\mathcal W$ be any variety of groups. The asssociated localization functor $L$ in the category of groups
is then right exact and thus conditionally flat.
\end{theorem}

\begin{proof}
Let $W$ be the set of words defining $\mathcal W$.
By Proposition~\ref{prop:halfexact} it is enough to prove that for any extension $1 \rightarrow K \rightarrow E \xrightarrow{p} G \rightarrow 1$,
the sequence $K/WK \rightarrow E/WE \rightarrow G/WG \rightarrow 1$ is exact. The localization $G \rightarrow G/WG$ is surjective,
hence so is $\bar p:E/WE \rightarrow G/WG$. We only need to identify the classes  of the form $eWE,$ for $e\in E,$ in the kernel of this 
last  morphism $\bar p$, which means
that $p(e) \in  WG$. In other words $p(e)$ can be written as a product $\gamma$ of conjugates of words $w(g_i)$ with $w \in W$. Since $p$ is surjective
there is a product $\epsilon$ of conjugates of the same words $w(e_i)$ whose image under $p$ is $p(e) = \gamma$. Therefore,
$e$ and $\epsilon$ differ by an element $k$ in the kernel $K$. But now, since $\epsilon \in WE$, we have
\[
eWE = e\epsilon^{-1} WE = k WE
\]
which proves exactness at $E/WE$.
\end{proof}

Since nilpotency is described by a variety of groups, we obtain  the following  result:

\begin{corollary}
\label{cor:nilpotent}
The localization functor in the category of groups
taking a group $G$ to its quotient  $G/\Gamma_c(G)$ by the lower central series is conditionally flat. 
\hfill{\qed}
\end{corollary}

\begin{remark}
\label{rem:classify}
{\rm The classifying space functor $B: \text{Groups} \to \text{Spaces}_*$ takes a discrete group $G$ to
the Eilenberg--Mac Lane space $BG$. Let $\varphi$ be a group homomorphism such that the localization
functor $L_\varphi$ is conditionally flat in the category of groups, but it not a nullification (for example the
above quotients by a given term of the lower central series). The homotopical localization $L_{B \varphi}$
associated to the corresponding map of classifying spaces is not conditionally flat as we know from
Theorem~\ref{thm:nullification}. However extensions of $\varphi$-local groups yield fibration sequences
of $B\varphi$-local classifying spaces and pull-back of such fibration sequences along any map of classifying
spaces are $L$-flat.
}
\end{remark}

\section{Examples, counter examples, and open questions}
\label{sec:negative}
Localization functors in the category of groups associated to varieties of groups or nullification functors
are conditionally flat.
However, as soon as the localization we consider is not one corresponding to a variety, things can easily go ``wrong".
Let us construct various counter examples.

\subsection{Epireflections and quasi-varieties}
A localization functor $L$ in the category of groups is called an \emph{epireflection}
if the localization morphism $G \to LG$ is always an epimorphism. Such localization functors
are in one to one correspondence with subfunctors of the identity, usually called \emph{radicals}
since, to an epireflection one associates the radical $R_L$ defined by $R_L(G) = \hbox{\rm Ker}(G \rightarrow LG)$.
A good source for the group theorist's point of view on radicals is Robinson's book \cite{MR48:11314}.

A localization functor in the category of groups is an epireflection if and only if there exists an epimorphism
$\varphi$ such that $L_\varphi$ is $L$. Thus every variety of groups $\mathcal W$ determines an epireflection, but
we will see that being an epireflection is not enough for conditional flatness, compare with \cite[Proposition~5]{MR1758736}.

\begin{theorem}
\label{thm:epireflection}
There are epireflections $L$ which are not conditionally flat.
\end{theorem}

\begin{proof}
Let $\phi: C_4 \to C_2$ be the projection and choose $L=L_\phi$. Any torsion-free group is local with respect to this
epireflection since there are no non-trivial morphism from a torsion group to a torsion-free group. Moreover
the cyclic group of order $2$ is local as well (it is the localization of~$C_4$).

Therefore the abelian group extension $\mathbf Z \to \mathbf Z \to \mathbf Z/2$ is an extension of local groups.
Let us pull it back along $\varphi$ itself. The pull-back $P$ is an extension of $\mathbf Z$ by $\mathbf Z/2$, which
must be trivial, so $P$ is isomorphic to $\mathbf Z \times \mathbf Z/2$, another local group! The pull-back extension
$\mathbf Z \to \mathbf Z \times \mathbf Z/2 \to \mathbf Z/4$ is therefore not preserved by $L$.
\end{proof}

\begin{remark}
\label{rmk:quasivariety}
{\rm Localization with respect to $C_4 \to C_2$ is an epireflection, and even better a localization associated to a
so-called \emph{quasi-variety}, \cite{MR0282908}. Whereas for a variety one imposes certain words to become trivial, in a quasi-variety one
does so provided certain equations are satisfied. In the previous proof the condition is that $x^4 = 1$. If so, then one
imposes $x^2 = 1$. We have thus actually proven a little bit more than stated in Theorem~\ref{thm:epireflection}: There
are epireflections associated to quasi-varieties which are not conditionally flat.
}
\end{remark}

\subsection{Other localization functors}
We turn now to a general localization functor and study which are the features which allow for
the construction of a non-$L$-flat pull-back from an $L$-flat group extension.
What is the general principle which lies behind this compatibility between pulling back and localizing?
Since nullification functors are known to be conditionally flat, we discard them and work from now on with a localization functor
which is not of the form $P_A$.

\begin{lemma}
\label{lem:notnullification}
Let $L_\phi$ be a localization functor which is not a nullification. Then there exists a non-identity localization morphism
$G \to L_\phi G$ which has $L_\phi$-local kernel.
\end{lemma}

\begin{proof}
Let $L_{\mathcal E (\phi)}$ be the universal epireflection associated to $L_\phi$, \cite[Theorem~8]{MR1758736},
which means that the localization morphism $G \rightarrow L_\phi G$ factors as
\[
G \twoheadrightarrow L_{\mathcal E (\phi)} G \hookrightarrow L_\phi G
\]
As $L_\phi$ is not a nullification functor by assumption, we have to deal with two cases. In the first one,
the epireflection is a nullification, and then there exists a ${\mathcal E (\phi)}$-local group $G$ such that
$G \rightarrow L_\phi G$ is injective, hence has a local kernel. In the second one the epireflection is not
a nullification, i.e. there exists a group $G$ such that the kernel of $G \twoheadrightarrow L_{\mathcal E (\phi)} G$
is not acyclic, \cite{MR1676617}. Fiberwise localization then yields a morphism $\overline{G} \to L_\phi G$ with
(non-trivial) local kernel.
\end{proof}

The previous lemma justifies the choice of $G$ in the following proposition.

\begin{proposition}
\label{prop:counterexample}
Let $f: A \to B$ be a group homomorphism and let $L = L_f$. Assume that there exist a non-identity localization morphism
$G \to LG$ with local kernel and a surjection $E \to LG$ from a local group $E$ such that $\Hom A E = \{ 1\} = \Hom B E$.
Then the pull-back $P$ of the diagram $E \to LG \leftarrow G$ is local. In particular $P \to E$ is not the localization morphism
and $L$ is not conditionally flat.
\end{proposition}

\begin{proof}
Any morphism from $A$, respectively $B$, to $P$ is given by a pair of
compatible morphisms to $G$ and $E$. By assumption the morphism to $E$ is trivial so that the morphism to $G$
must factorize through the kernel of the localization, which is local. Therefore $\Hom A P = \{1\} = \Hom B P$.
\end{proof}

This construction helps to find many localization functors which are not conditionally flat. The first occurence
of such a localization was the epireflection associated to a quasi-variety encountered in the proof of
Theorem~\ref{thm:epireflection}.

\begin{example}
\label{ex:assaf}
{\rm Let $f: A_n \hookrightarrow A_{n+1}$ be Libman's localization morphism in \cite{MR1760593} for $n\geq 7$. Pick
a (free) presentation $F_1 \to F_0 \to A_{n+1}$. Any free group is obviously $f$-local, so that the proposition applies
and the pull-back of $A_n \to A_{n+1} \leftarrow F_1$ is also $f$-local.
}
\end{example}

\subsection{What about non-localization functors?}
If we consider a functor which behaves well with push-outs and more generally colimits, and (therefore)
badly with respect to extensions and pull-backs, we will see that it has very little chance to be conditionally flat. For any group $G$ we write $S_p(G) \subset G$ for the subgroup generated by its elements of order $p$.

\begin{proposition}
\label{prop:counter}
The functor $S_p$ is not conditionally flat.
\end{proposition}

\begin{proof}
We exhibit a counter-example: Set $p=2$ and consider the (central) extension $\Z/2 \to D_8 \to \Z/2 \times \Z/2$
where $D_8$ is the dihedral group of order $8$, given by the standard presentation $<x, y\, | \, x^4=y^2 = 1 = yxyx >$. This extension
is $S_2$-flat since $D_8$ is generated by $y$ and $yx$, both elements of order $2$. However if we pick in the
base $\Z/2 \times \Z/2$ the copy of $\Z/2$ generated by the image of $x$ and pull the extension back along this inclusion,
we get the extension $\Z/2 \to \Z/4 \to \Z/2$ which is not $S_2$-flat.
\end{proof}

In fact we could also have chosen the analogous property defined by replacing the subgroup $S_p(-)$ by $\mathbf Z/p$-cellularization.
The class of $\mathbf Z/p$-cellular groups is closed under colimits and the question we ask deals with extensions and pull-backs.
This is why we should not expect them to behave well together. One should maybe rather ask the dual question about
the interplay of push-outs and cellularization.

\subsection{Open questions}
We know now that general group localization functors do not behave as nicely as one could expect with respect
to pulling back extensions, not even for abelian groups! Nullifications and epireflections coming from group varieties
are the only one we know of that behave well. We have not dealt with localization functors $L$ for which $G \to LG$ is
not surjective, such as localization at a set of primes.

\medskip

\noindent
{\bf Question A.}  Are there conditionally flat localization functors which are not eprireflections?

\medskip

Notice that rationalization, and localization at a set of primes, are exact functors in the category of abelian
groups. They are therefore flat, hence conditionally flat in the category of abelian groups.

\medskip

\noindent
{\bf Question B.} Is rationalization, i.e. localization with respect to multiplication by $p$ on the integers for all prime numbers $p$,
conditionally flat in the category of groups?

\medskip

By moving from the category of spaces to that of groups, we found more conditionally flat localization functors. By restricting
even more to a strict subcategory of groups, the class of conditionally flat functors will increase. 

\medskip

\noindent
{\bf Question C.} What does conditional flatness mean in a full subcategory of groups, such as abelian or nilpotent groups?



\bibliographystyle{amsplain}
\def\cprime{$'$}
\providecommand{\bysame}{\leavevmode\hbox to3em{\hrulefill}\thinspace}
\providecommand{\MR}{\relax\ifhmode\unskip\space\fi MR }
\providecommand{\MRhref}[2]{%
  \href{http://www.ams.org/mathscinet-getitem?mr=#1}{#2}
}
\providecommand{\href}[2]{#2}



\end{document}